\documentclass[a4paper,twoside,12pt]{article}

\usepackage{amssymb,amsmath,amsthm,latexsym}
\usepackage{amsfonts}
\usepackage{graphicx,caption,subcaption}
\usepackage[pdftex,bookmarks,colorlinks=false]{hyperref}
\usepackage[hmargin=1.2in,vmargin=1.2in]{geometry}
\usepackage[mathscr]{euscript}
\usepackage[auth-lg,affil-sl]{authblk}

\usepackage[auth-lg,affil-sl]{authblk}
\allowdisplaybreaks

\usepackage{float}
\usepackage{calc,tikz}
\usetikzlibrary{arrows,decorations.markings}
\tikzstyle{vertex}=[circle, draw, inner sep=1pt, minimum size=8pt]

\usepackage{multirow}
\usepackage{hhline}
\usepackage{booktabs}
\usepackage{longtable}

\newcommand{\noi}{\noindent}

\newtheorem{theorem}{Theorem}[section]
\newtheorem{definition}[theorem]{Definition}
\newtheorem{lemma}[theorem]{Lemma}

\newtheorem{corollary}[theorem]{Corollary}

\newtheorem{conjecture}{Conjecture}
\newtheorem{problem}{Problem}

\title{\textbf{\sc Some New Results on Proper Colouring of Edge-set Graphs}}

\author{Johan Kok$^\ast$, Sudev Naduvath$^\dagger$}
\affil{\small Centre for Studies in Discrete Mathematics\\ Vidya Academy of Science \& Technology \\Thalakkottukara, Thrissur, Kerala, India.\\ $^\ast${\tt kokkiek2@tshwane.gov.za}\\$^\dagger${\tt sudevnk@gmail.com}}

\date{}

\begin{document}
\maketitle

\begin{abstract}
\noi In this paper, we present a foundation study for proper colouring of edge-set graphs. The authors consider that a detailed study of the colouring of edge-set graphs corresponding to the family of paths is best suitable for such foundation study. The main result is deriving the chromatic number of the edge-set graph of a path, $P_{n+1}$, $n \geq 1$. It is also shown that edge-set graphs for paths are perfect graphs.
\end{abstract}

\noi \textbf{Keywords:}  Chromatic colouring, rainbow neighbourhood, rainbow neighbourhood number, edge-set graph.
\vspace{0.35cm}

\noi \textbf{Mathematics Subject Classification 2010:} 05C15, 05C38, 05C75, 05C85. 

\section{Introduction}

For general notation and concepts in graphs and digraphs see \cite{BM,FH,DBW}. Unless mentioned otherwise, all graphs we consider in this paper are finite, simple, connected and undirected  graphs.

For a set of distinct colours $\mathcal{C}= \{c_1,c_2,c_3,\ldots,c_\ell\}$, a \textit{vertex colouring} of a graph $G$ is an assignment $\varphi:V(G) \mapsto \mathcal{C}$ . A vertex colouring is said to be a \textit{proper vertex colouring} of a graph $G$ if no two distinct adjacent vertices have the same colour. The cardinality of a minimum set of colours in a proper vertex colouring of $G$ is called the \textit{chromatic number} of $G$ and is denoted $\chi(G)$. A colouring of $G$ with exactly $\chi(G)$ colours may be called a $\chi$-colouring or a \textit{chromatic colouring} of $G$. 

A \textit{minimum parameter colouring} of a graph $G$ is a proper colouring of $G$ which consists of the colours $c_i;\ 1\le i\le \ell$, with minimum possible values for the subscripts $i$. Unless stated otherwise, we consider minimum parameter colouring throughout this paper. 

The set of vertices of $G$ having the colour $c_i$ is said to be the \textit{colour class} of $c_i$ in $G$ and is denoted by $\mathcal{C}_i$. The cardinality of the colour class $\mathcal{C}_i$ is said to be the weight of the colour $c_i$, denoted by $\theta(c_i)$. Note that $\sum\limits_{i=1}^{\ell}\theta(c_i)=\nu(G)$. 

Unless mentioned otherwise,  we colour the vertices of a graph $G$ in such a way that $\mathcal{C}_1=I_1$, the maximal independent set in $G$, $\mathcal{C}_2=I_2$, the maximal independent set in $G_1=G-\mathcal{C}_1$ and proceed like this until all vertices are coloured. This convention is called \textit{rainbow neighbourhood convention} (see \cite{KSJ}).

Unless mentioned otherwise, we shall colour a graph in accordance with the rainbow neighbourhood convention. 

Note that the closed neighbourhood $N[v]$ of a vertex $v \in V(G)$ which contains at least one coloured vertex of each colour in the chromatic colouring, is called a rainbow neighbourhood see \cite{KSJ}. The number of vertices in $G$ which yield rainbow neighbourhoods,  denoted by $r_\chi(G)$, is called the \textit{rainbow neighbourhood number} of $G$. Further studies on the rainbow neighbourhood number of different graph classes and graph operations can be seen in \cite{KSJ,KN1,KN2,KSB,NSKK,SNKK}.

In \cite{KSJ}, the bounds on $r_\chi(G)$ corresponding to of minimum proper colouring, denoted by $r^-_\chi(G)$ and $r^+_\chi(G)$, have been defined as the minimum value and maximum value of $r_\chi(G)$ over all permissible colour allocations. If we relax connectedness, it follows that the null graph $\mathfrak{N}_n$ of order $n\geq 1$ has $r^-(\mathfrak{N}_n)=r^+(\mathfrak{N}_n)=n$. For bipartite graphs and complete graphs, $K_n$ it follows that, $r^-(G)=r^+(G)=n$ and $r^-(K_n)=r^+(K_n)=n$.

We observe that if it is possible to permit a chromatic colouring of any graph $G$ of order $n$ such that the star subgraph obtained from vertex $v$ as center and its open neighbourhood $N(v)$ the pendant vertices, has at least one coloured vertex from each colour for all $v \in V(G)$ then $r_\chi(G)=n$. Certainly, examining this property for any given graph is complex.

\begin{lemma}
{\rm \cite{KSJ}} For any graph $G$ the graph $G'=K_1+G$ has $r_\chi(G')=1+r_\chi(G)$.
\end{lemma}

\section{Rainbow Neighbourhood Number of Edge-set Graphs}

Edge-set graphs were introduced in \cite{KSC}. As the notion of an edge-set graph seems to be largely unknown. Therefore, the main definition and some important observations from \cite{KSC} will be presented in this section.

Let $A$ be a non-empty finite set. Let the set of all $s$-element subsets of $A$ (arranged in some order), where $1\leq s\leq |A|$, be denoted by $\mathscr{S}$ and the $i$-th element of $\mathscr{S}$ by, $A_{i,s}$.
 
\begin{definition}\label{Defn-2.1}{\rm  
\cite{KSC} Let $G(V,E)$ be a non-empty finite graph with $|E|=\varepsilon \geq 1$ and $\mathscr{E}=\mathscr{P}(E)-\{\emptyset\}$, where $\mathscr{P}(E)$ is the power set of the edge set $E(G)$. For $1\leq s\leq \varepsilon$, let $\mathscr{S}$ be the collection of all $s$-element subsets of $E(G)$ and $E_{s,i}$ be the $i$-th element of $\mathscr{S}$. Then, the \textit{edge-set graph} corresponding to $G$, denoted by $\mathcal{G}_G$, is the graph with the following properties.
\begin{enumerate}\itemsep0mm
\item[(i)] $|V(\mathcal{G}_G)|=2^{\varepsilon}-1$ so that there exists a one to one correspondence between $V(\mathcal{G}_G)$ and $\mathscr{E}$;
\item[(ii)] Two vertices, say $v_{s,i}$ and $v_{t,j}$, in $\mathcal{G}_G$ are adjacent if some elements (edges of $G$) in $E_{s,i}$ is adjacent to some elements of $E_{t,j}$ in $G$.
\end{enumerate}
}\end{definition}

From the above definition, it can be seen that the edge-set graph $\mathcal{G}_G$ of a given graph $G$ is dependent not only on the number of edges $\varepsilon$, but the structure of $G$ also. Note that it was erroneously remarked in \cite{KSC} that non-isomorphic graphs of the same size have distinct edge-set graphs. Figure \ref{fig:Figure-2EdgeSet} illustrates one contradictory case.

Note that an edge-set graph $\mathcal{G}_G$ has an odd number of vertices. If $G$ is a trivial graph, then $\mathcal{G}_G$ is an empty graph (since $\varepsilon=0$). Also, $\mathcal{G}_{P_2} = K_1$ and $\mathcal{G}_{P_3}=C_3$. In \cite{KSC} the following conventions were used.

\begin{enumerate}\itemsep0mm
\item[(i)] If an edge $e_j$ is incident with vertex $v_k$, then we write it as $(e_j \to v_k)$.
\item[(ii)] If the edges $e_i$ and $e_j$ of a graph $G$ are adjacent, then we write it as $e_i\sim e_j$.
\item[(iii)] The $n$ vertices of the path $P_n$ are positioned horizontally and the vertices and edges are labeled from left to right as $v_1, v_2, v_3, \ldots, v_n$ and $e_1, e_2, e_3, \ldots, e_{n-1}$, respectively.
\item[(iv)] The $n$ vertices of the cycle $C_n$ are seated on the circumference of a circle and the vertices and edges are labeled clockwise as $v_1, v_2, v_3, \ldots, v_n$ and $e_1, e_2, e_3, \ldots, e_n$, respectively such that $e_i=v_iv_{i+1}$, in the sense that $v_{n+1}=v_1$.
\end{enumerate}

Invoking the definition and observations given above, it is noticed that both $d^t_{G(e)}(G)$ and $d_{G(e)}(v_i)$ are single values, while $d_{G(v_k)}(e_j) \leq d_{G(v_m)}(e_j), (e_j \to v_k), (e_j \to v_m)$. The graphs having three edges $e_1, e_2, e_3$ are graphs $P_4, C_3$, and $K_{1,3}$. The corresponding edge-set graphs on the vertices $v_{1,1} =\{e_1\}, v_{1,2} = \{e_2\}, v_{1,3} = \{e_3\}, v_{2,1}=\{e_1,e_2\}, v_{2,2}=\{e_1, e_3\}, v_{2,3}=\{ e_2,e_3\}, v_{3,1}=\{e_1,e_2,e_3\}$ are depicted below.

\noi Figure \ref{fig:Figure-1EdgeSet} depicts the edge-set graph $\mathcal{G}_{P_4}$.

\begin{figure}[h!]
\centering
\includegraphics[width=0.5\linewidth]{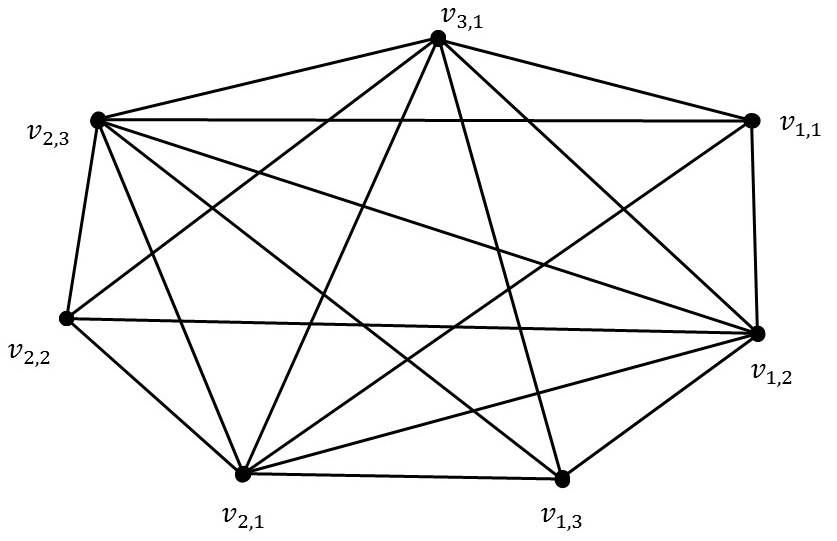}
\caption{Edge-set graph $\mathcal{G}_{P_4}$.}
\label{fig:Figure-1EdgeSet}
\end{figure}

\noi Figure \ref{fig:Figure-2EdgeSet} depicts the edge-set graph $\mathcal{G}_{C_3} = \mathcal{G}_{K_{1,3}} = K_7$.

\begin{figure}[h!]
\centering
\includegraphics[width=0.5\linewidth]{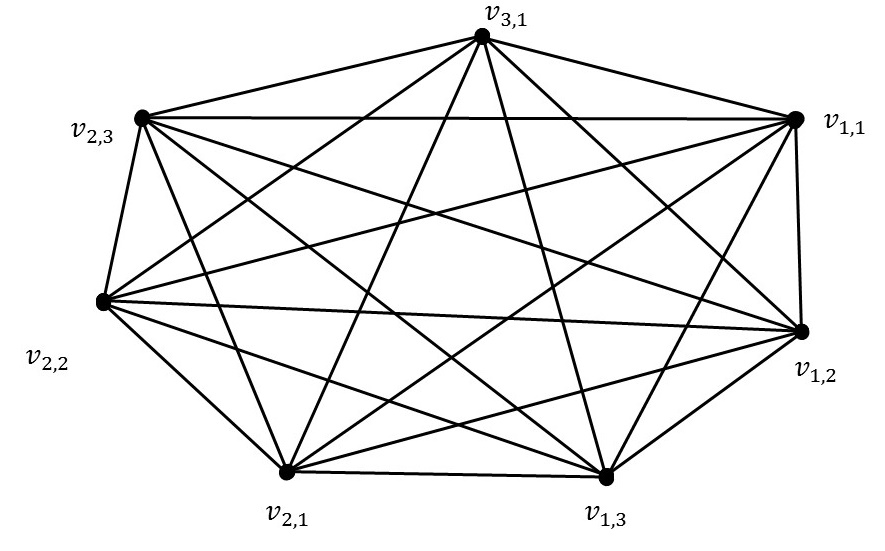}
\caption{Edge-set graph $\mathcal{G}_{C_3} = \mathcal{G}_{K_{1,3}} = K_7$.}
\label{fig:Figure-2EdgeSet}
\end{figure}

Notice that both $\mathcal{G}_{C_3}$ and $\mathcal{G}_{K_{1,3}}$ are complete graphs. 

\section{Proper Colouring of the Edge-set Graphs of Paths}

It is known that for a given size $\varepsilon \geq 1$ a graph of maximum order $\nu$, is a tree. Hence, for a given size the graphs with maximum structor index $si(G)$ (see \cite{KSB}), are the corresponding trees, $T$. It easily follows that for $\varepsilon(T) \geq 3$ only the star graphs have $\mathcal{G}_{S_{\varepsilon +1}}$, complete. Put another way, a tree $T$ has $\mathcal{G}_T$ complete if and only if $diam(T) \leq 2$. From the family of trees, a path corresponding to a given $\varepsilon$, denoted by $P_\varepsilon$, has largest diameter. These observations motivate a detailed study of the proper colouring and associated colour parameters of edge-set graphs of paths to lay the foundation for studying more complex graph classes.

For this section paths of the form $P_{n+1} = v_1e_1v_2e_2v_3\cdots e_nv_{n+1}$, will be considered. Such graph will be abbreviated to $P_{n+1} = v_1e_iv_i\succ$, $1\leq i \leq n$. To easily relate the results with Definition \ref{Defn-2.1}, note that $\varepsilon(P_{n+1})=n$. It can be easily verified that $\mathcal{G}_{P_2}=K_1$. Hence, $\chi({G}_{P_2})=1$. Also, $\mathcal{G}_{P_3}=K_3$ and hence, $\chi(\mathcal{G}_{P_3})=3$. These observations bring the main results. First, we state an important lemma.

\begin{lemma}\label{Lem-3.1}
Let $G(V,E)$ be a non-empty finite graph with $|E|=\varepsilon \geq 1$ and $\mathscr{E}=\mathscr{P}(E)-\{\emptyset\}$, where $\mathscr{P}(E)$ is the power set of the edge set $E(G)$. Then each edge $e_i$ is in exactly $2^{\varepsilon -1}$ subsets of $\mathscr{E}$.
\end{lemma}
\begin{proof}
The result follows directly from the well-definedness and well-ordering of the power set, $\mathscr{P}(E)$.
\end{proof}

It is observed that if the number of subsets which has say, $e_i$ as element is $t$, then within the corresponding $t$ subsets the edge $e_j$, $j\neq i$ will be in $\frac{t}{2} = 2^{\varepsilon -2}$ of those subsets.

\begin{theorem}\label{Thm-3.2}
The edge-set graph $\mathcal{G}_{P_{n+1}}$, $n \geq 1$ has
\begin{equation*}
\chi(\mathcal{G}_{P_{n+1}}) =
\begin{cases}
1\ \text{or}\ 3, & \text {if $P_2$ or $P_3$ respectively},\\
5, & \text {if $P_4$},\\
2^{n-1}+ 2^{n-2}-2, & \text {for $P_{n+1},\ n \geq 4$}.
\end{cases}
\end{equation*}
\end{theorem}
\begin{proof}
\textit{Part 1:} Trivial is the observation that $\mathcal{G}_{P_2} = K_1$ and that result in equality. It has been observed that $\mathcal{G}_{P_3} = K_3$ and hence $\chi(\mathcal{G}_{P_3})=3$.

\textit{Part 2:} In constructing $\mathcal{G}_{P_4}$ begin with $\mathcal{G}_{P_3}$ which has vertices $\{e_1\}$, $\{e_2\}$, $\{e_1,e_2\}$. Add a disjoint copy of $\mathcal{G}_{P_3}$ and relabel the vertices of this copy to be $\{e_1,e_3\}$, $\{e_2,e_3\}$, $\{e_1,e_2,e_3\}$ to obtain, $\mathcal{G}'_{P_3}$. Clearly, $\mathcal{G}'_{P_3}$ complies with Definition \ref{Defn-2.1}.

Consider $H= \mathcal{G}_{P_3}\cup \mathcal{G}'_{P_3}$ and add the cut edges, $\{e_2\}\{e_1,e_3\}$, $\{e_2\}\{e_2,e_3\}$, $\{e_2\}\{e_1,e_2,e_3\}$, $\{e_1,e_2\}\{e_1,e_3\}$, $\{e_1,e_2\}\{e_2,e_3\}$, $\{e_1,e_2\}\{e_1,e_2,e_3\}$. Clearly, the induced subgraph, $\langle\{e_2\},\{e_1,e_2\},\{e_1,e_3\},\{e_2,e_3\},\{e_1,e_2,e_3\}\rangle = K_5$. Now add all additional bridges in accordance with Definition \ref{Defn-2.1} to obtain graph $H'$. Due to symmetry considerations between edges $e_1$ and $e_2$ in $P_3$, exactly two maximum cliques $K_5$ come into existence hence, $\omega(H') = 5$. Finally, by adding vertex $\{e_3\}$ and the corresponding edges in accordance with Definition \ref{Defn-2.1} and by symmetry considerations between edges $e_1$ and $e_3$ in $P_4$, the edge-set graph $\mathcal{G}_{P_4}$ has exactly four maximum cliques $K_5$. Therefore, $\chi(\mathcal{G}_{P_4}) \geq 5$.

Invoking Definition \ref{Defn-2.1}, consider the following colouring of $\mathcal{G}_{P_4}$. Let $c(v_{1,1})=c_1$, $c(v_{1,3})=c_1$, $c(v_{2,2})=c_1$, $c(v_{1,2})=c_2$, $c(v_{2,1})=c_3$, $c(v_{2,3})=c_4$, $c(v_{3,1})=c_5$.  Clearly, the colouring is proper and hence $\chi(\mathcal{G}_{P_4}) \leq 5$. Hence we have $\chi(\mathcal{G}_{P_4}) = 5$.

\textit{Part 3:} For $n\geq 4$, and the path path $P_{n+1}$ the edge-set graph $\mathcal{G}_{P_{(n-1)+1}}$ of the preceding path hence, the $(n-1)$-edge path $P_{(n-1)+1}$, is incomplete. In accordance with the procedure described in Part 2, consider $\mathcal{G}_{P_{(n-1)+1}}$ and $\mathcal{G}'_{P_{(n-1)+1}}$. Since in $\mathcal{G}'_{P_{(n-1)+1}}$ the edge $e_n$ has been added to each vertex corresponding to the vertices $v_{i,j} \in V(\mathcal{G}_{P_{(n-1)+1}})$, the new edges in accordance with Definition \ref{Defn-2.1} are those between all pairs of vertices for which at least one vertex has $e_{n-1} \in v'_{i,j}$. From Lemma \ref{Lem-3.1}, it follows that at least one complete induced subgraph, $K_{2^{n-2}}$ exists in $\mathcal{G}'_{P_{(n-1)+1}}$. All pairs of vertices which has both $e_{n-2},e_{n-1} \in v'_{i,j}$ is an edge in $\mathcal{G}'_{P_{(n-1)+1}}$ so  least one complete induced subgraph, $K_{2^{n-2}+1}$ exists in $\mathcal{G}'_{P_{(n-1)+1}}$. Proceeding to vertices for which edge $e_{n-3} \in v'_{i,j}$ and so on until the edge $e_1$ has been accounted for results in $\mathcal{G}'_{P_{(n-1)+1}}$ being complete. Hence, $\chi(\mathcal{G}'_{P_{(n-1)+1}}) = 2^{n-1}-1$.

Finally, by adding the bridges between $\mathcal{G}_{P_{(n-1)+1}}$ and $\mathcal{G}'_{P_{(n-1)+1}}$ and through similar arguments in respect of edges $e_{n-2},e_{n-1} \in v_{i,j} \in V(\mathcal{G}_{P_{(n-1)+1}})$ and so on, it follows that at least one maximum induced clique, of order $2^{n-2}-1 + \chi(\mathcal{G}_{P_{(n-1)+1}})$, exists in $\mathcal{G}_{P_{n+1}}$. Therefore, $\chi(\mathcal{G}_{P_{n+1}}) \geq 2^{n-1}+ 2^{n-2}-2$. By allocating colours similar to the procedure described in Part-2, it follows that $2^{n-1}+ 2^{n-2}-2 \leq \chi(\mathcal{G}_{P_{n+1}}) \leq 2^{n-1}+ 2^{n-2}-2 \Leftrightarrow \chi(\mathcal{G}_{P_{n+1}}) = 2^{n-1}+ 2^{n-2}-2$. Therefore, by immediate induction, the result follows for all $n \geq 4$. 
\end{proof}

\begin{corollary}\label{Cor-3.3}
\begin{enumerate}\itemsep0mm
\item[(a)] Each vertex in an edge-set graph $\mathcal{G}_{P_{n+1}}$, $n \geq 2$ belongs to some maximum clique in $\mathcal{G}_{P_{n+1}}$.
\item[(b)] The edge-set graphs $\mathcal{G}_{P_{n+1}}$, $n \geq 1$ has clique number, $\omega(\mathcal{G}_{P_{n+1}}) = 2^{n-1}+ 2^{n-2}-2$.
\item[(c)] The edge-set graphs $\mathcal{G}_{P_{n+1}}$, $n \geq 1$ are perfect graphs.
\item[(d)] The edge-set graph $\mathcal{G}_{P_{n+1}}$ has, $r^-_chi(\mathcal{G}_{P_{n+1}}) = r^+_chi(\mathcal{G}_{P_{n+1}}) = 2^n - 1$.
\end{enumerate}
\end{corollary}
\begin{proof}
The results are a direct consequence from the proof of Theorem \ref{Thm-3.2}.
\end{proof}

\begin{theorem}\label{Thm-3.4}
An edge-set graph $\mathcal{G}_{P_{n+1}}$, $n \geq 1$  is a perfect graph.
\end{theorem}
\begin{proof}
For $P_1$, $P_2$ the result is trivial. From Theorem \ref{Thm-3.2} and Corollary \ref{Cor-3.3}(b) we have, $n\geq 2$ and hence it follows that $\omega(\mathcal{G}_{P_{n+1}}) =2^{n-1}+ 2^{n-2}-2 = \chi(\mathcal{G}_{P_{n+1}})$. Hence, an edge-set graph is weakly perfect. From Definition \ref{Defn-2.1}, it follows that an edge-set graph has a unique maximum independent set $X$. Furthermore, $\langle X\rangle$ is a null graph hence, any subgraph thereof is perfect.

Also, from Corollary \ref{Cor-3.3}(a), each vertex in $V(\mathcal{G}_{P_{n+1}})$ is in some induced maximum clique. It then follows that $\omega(H) = \chi(H)$, $\forall H\subseteq \mathcal{G}_{P_{n+1}}$, $n \geq 1$. Hence the result. 
\end{proof}

\begin{conjecture}
The edge-set graphs of acyclic graphs are perfect graphs.
\end{conjecture}

\section{Conclusion}

\textbf{Research problem:} The notion of a chromatic core subgraph of a graph $G$ was introduced in \cite{KSB}. We recall that, for a graph $G$ its \textit{structural size} is measured by its \textit{structor index} denoted and defined as, $si(G) = \nu(G) + \varepsilon(G)$. We say that the smaller of graphs $G$ and $H$ is the graph satisfying the condition, $min\{si(G), si(H)\}$. If $si(G) = si(H)$ the graphs are of equal structural size but not necessarily isomorphic. A straight forward example is the path, $P_4$ and the star graph, $S_3$.

\begin{definition}{\rm 
For a finite, undirected simple graph $G$ of order $\nu(G) = n \geq 1$ a chromatic core subgraph $H$ is a smallest induced subgraph $H$ $($smallest in respect of $si(H)$$)$ such that, $\chi(H) = \chi(G)$.
}\end{definition}

From the construction used in the proof of Theorem \ref{Thm-3.2} it follows that a finite number of distinct maximum cliques can be associated with a given edge-set graph $\mathcal{G}_{P_{n+1}}$. As an application, the largest number of vertices common to the maximum number of chromatic core subgraphs can be considered the most strategic vertices for protection from a disaster management and recovery plan in the event of graph destruction. The aforesaid observation motivates us to introduce a new graph parameter called the \textit{chromatic cluster number} of a graph $G$. It is denoted by $\complement(G)$. From Theorem \ref{Thm-3.2} it follows that $\complement (\mathcal{G}_{P_2}) = \complement (\mathcal{G}_{P_3}) =1$ and $\complement (\mathcal{G}_{P_4}) = 4$. Note that the vertices $v_{1,1} = \{e_1\}$, $v_{1,3} = \{e_3\}$, $v_{2,2} = \{e_1,e_3\}$ and $v_{1,3} = \{e_1,e_2,e_3\}$ corresponds to $\complement (\mathcal{G}_{P_4})$. 

\begin{problem}{\rm 
For the edge-set graph $\mathcal{G}_{P_{n+1}}$, $n \geq 4$, determine  $\complement(\mathcal{G}_{P_{n+1}})$.
}\end{problem}

The research on set-graphs [3] and edge-set graphs naturally leads to new concepts such as vertex degree sequence set-graphs and colour set-graphs and colour-string set-graphs. Preliminary definitions are provided below.

\subsection*{Terms of reference} 

\begin{enumerate}
\item If the degree sequence of a graph $G$ of order $n\geq 1$ is $(d_1\leq d_2\leq d_3\leq \cdots, d_n)$, then for a subsequence $(d_{t+1} = d_{t+2} = \cdots = d_{t+\ell} = m_i)$, $t\geq 0$, $1\leq \ell \leq n$, label the corresponding vertices to be $m_{i,1},m_{i,2},m_{i,3},\ldots,m_{i,\ell}$. Consider the set $\mathcal{V}(G) = \mathscr{P}(V) -\emptyset$ where, $\mathscr{P}(V)$ is the power set of $V(G)$.

\begin{definition}\label{Defn-4.2}
The degree sequence set-graph corresponding to $G$, denoted by $\mathcal{G}_{\mathcal{V}(G)}$, is the graph with the following properties.
\begin{enumerate}\itemsep0mm 
\item[(i)] $|\mathcal{G}_{\mathcal{V}(G)}|=2^{\nu}-1$ so that there exists a one to one correspondence between $V(\mathcal{G}_{\mathcal{V}(G)})$ and $\mathcal{V}(G)$.
\item[(ii)] Two vertices, say $v_{s,i}$ and $v_{t,j}$, in $\mathcal{G}_{\mathcal{V}(G)}$ are adjacent if some element(s) (specific vertex degree(s) of $G$) in $v_{s,i}$ is adjacent to some element(s) of $v_{t,j}$ in $G$.
\end{enumerate}
\end{definition}

It follows easily that for a complete graph $K_n$, $n\geq 1$ has its corresponding degree sequence set-graph, a complete graph.

\begin{problem}{\rm 
Discuss the properties of the degree sequence set-graph corresponding to graph $G$.
}\end{problem}

\item Let the minimum colour set $\mathcal{C} =\{c_1,c_2,c_3,\ldots, c_\chi\}$ permit a chromatic colouring of $G$ in accordance with the rainbow neighbourhood convention. Let $\mathcal{C}^{\{\}}(G) = \mathscr{P}(\mathcal{C}) - \emptyset$ where, $\mathscr{P}(\mathcal{C})$  is the power set of $\mathcal{C}$.

\begin{definition}\label{Defn-4.3}
The colour set-graph corresponding to $G$, denoted by $\mathcal{G}_{\mathcal{C}^{\{\}}(G)}$, is the graph with the following properties.
\begin{enumerate}\itemsep0mm
\item[(i)] $|\mathcal{G}_{\mathcal{C}^{\{\}}(G)}|=2^{\chi}-1$ so that there exists a one to one correspondence between $V(\mathcal{G}_{\mathcal{C}^{\{\}}(G)})$ and $\mathcal{C}^{\{\}}(G)$.
\item[(ii)] Two vertices, say $v_{s,i}$ and $v_{t,j}$, in $\mathcal{G}_{\mathcal{C}^{\{\}}(G)}$ are adjacent if some element(s) (specific vertex degree(s) of $G$) in $v_{s,i}$ is adjacent to some element(s) of $v_{t,j}$ in $G$.
\end{enumerate}
\end{definition}

Clearly, for all graphs $G$ with $\chi(G) =2$ the colour set-graph is $K_3$.

\begin{problem}
Discuss the properties of the colour set-graph corresponding to a chromatic colouring of a graph $G$.
\end{problem}

This problem is similar to (1). For a minimum colour set $\mathcal{C} =\{c_1,c_2,c_3,\ldots, c_\chi\}$ the corresponding colour weight sequence is $(\underbrace{c_1,c_1,c_1,\ldots,c_1}_{\theta(c_1)~entries},\cdots,\underbrace{c_\chi,c_\chi,c_\chi,\ldots,c_\chi}_{\theta(c_\chi~entries})$. 

Let $\mathcal{C}^\circ(G) = \{c_{1,1},c_{1,2},c_{1,3}\ldots,c_{1,\theta(c_1)},\cdots,c_{\chi,1},c_{\chi,2},c_{\chi,3},\ldots,c_{\chi,\theta(c_\chi)}\}$. We can define the colour-string set-graph, $\mathcal{G}_{\mathcal{C}^\circ(G)}$ similar to Definition \ref{Defn-4.2}.

\begin{problem}
Research the properties of the colour-string set-graph corresponding to a chromatic colouring of a graph $G$.
\end{problem}

\end{enumerate}

\end{document}